\RequirePackage{fix-cm}
\documentclass[smallextended, envcountsame]{svjour3}
\smartqed
\def\makeheadbox{\relax}
\journalname{}

\usepackage[unicode = false,
    pdftoolbar = true,
    colorlinks = true,
    linkcolor = blue,
    citecolor = blue,
    filecolor = black,
    urlcolor = blue,
    breaklinks = true]{hyperref}
\usepackage[ruled]{algorithm2e}
\usepackage[utf8]{inputenc}
\usepackage{listings}           
\usepackage{enumitem}
\usepackage{amsmath}            
\usepackage{amssymb}            
\usepackage{xcolor}             
\usepackage[colorinlistoftodos]{todonotes}
\usepackage{graphicx}
\usepackage[ruled]{algorithm2e}
\usepackage{verbatim}
\usepackage{geometry}
\usepackage{enumerate}
\newcommand{\R}{\mathbb{R}}

\newcommand{\D}{\mathbb{D}}
\newcommand{\B}{\mathbb{B}}
\newcommand{\G}{\mathbf{G}}
\newcommand{\A}{\mathbf{A}}
\newcommand{\V}{c\mathbf{V}}

\newcommand{\set}{\mathbb{S}}
\newcommand{\Pe}{\mathbb{P}}

\newcommand{\one}{\mathbf{1}}

\newcommand{\us}{u_{\ast}}

\newcommand{\p}{\mathring{p}_{\B}}

\DeclareMathOperator{\pspan}{pspan}

\DeclareMathOperator{\cm}{cm}

\DeclareMathOperator{\spann}{span}

\DeclareMathOperator{\argminop}{argmin}
\newcommand{\argmin}[1]{\underset{#1}{\argminop}}

\begin{document}
\title{\centering A Deterministic Algorithm to Compute the Cosine Measure of a Finite Positive Spanning Set}
\author{Warren Hare \and Gabriel Jarry-Bolduc}
\institute{Warren Hare \at
              Department of Mathematics, University of British Columbia, Kelowna, \\
              British Columbia, Canada.  Hare's research is partially funded by the Natural Sciences and Engineering Research Council (NSERC) of Canada, Discover Grant \#2018-03865.\\
              \email{warren.hare@ubc.ca}           
           \and
          Gabriel Jarry-Bolduc\at
           Department of Mathematics, University of British Columbia, Kelowna, \\
              British Columbia, Canada.  Jarry-Bolduc's research is partially funded through the Natural Sciences and Engineering Research Council (NSERC) of Canada, Discover Grant \#2018-03865.\\
              \email{gabjarry@alumni.ubc.ca}
              }
\date{\today}

\maketitle
\begin{abstract}
Originally developed in 1954, positive bases and positive spanning sets have been found to be a valuable concept in derivative-free optimization (DFO). The quality of a positive basis (or positive spanning set) can be quantified via the {\em cosine measure} and convergence properties of certain DFO algorithms are intimately linked to the value of this measure. However, it is unclear how to compute the cosine measure for a positive basis from the definition. In this paper, a deterministic algorithm to compute the cosine measure of any positive basis or finite positive spanning set is provided. The algorithm is proven to return the exact value of the cosine measure in finite time.
\end{abstract}

\section{Introduction}

Positive bases have been studied since the 1950s. The theory has been first developed by Davis in \cite{Davis1954} and McKinney in \cite{Mckinney1962}. In the last few decades, their popularity has drastically increased due to their value in DFO. The value of positive bases in derivative-free optimization was  revealed in 1996  \cite{Lewis1996} when it was shown that if the gradient of a function at a point exist and is nonzero, then there exists a vector $d$ in any positive basis (or any positive spanning set) such that $d$ is a descent direction of the function at that point.

 Since then, several  derivative-free algorithms using positive bases have been developed. More specifically, positive bases are employed in direct search methods such as pattern search \cite{Custodio2008,Torczon1997,Vaz2009}, generalised pattern search \cite{Audet2017}, grid-based methods \cite{Coope2001,Coope2002}, generating set search \cite{Kolda2003}, mesh adaptive  direct search \cite{Abramson2009,Audet2006,Audet2014,vandykeasaki2013}  and implicit filtering \cite{Kelley2011}.

 A handful of papers have focused on the theory behind positive bases and their characterization \cite{Davis1954,Mckinney1962,Regis2016}. Davis established that the maximal size $s$ of a positive basis is $s=2n$ \cite{Davis1954}. A shorter proof of this result was published in 2011 by Audet \cite{Audet2011}. Also, it is straightforward to show that the minimal size $s$ of a positive basis is $s=n+1$ \cite{Davis1954}.  Minimal and maximal positive bases are now well-understood and their structure can  be rigorously characterized  \cite{Regis2016}. A few results on intermediate positive bases ($n+1<s<2n$) can  be found in  \cite{Reay1965,Reay1966,Romanowicz1987,Shephard1971}.

 With regards to DFO, the key instrument to measure the quality of a positive basis is called cosine measure \cite{Kolda2003} (see also \cite{Torczon1997}). In general, having higher cosine measure is preferable and can be thought as covering the space more uniformly with the vectors  contained in the positive basis. However, no methods have been proposed thus far  to calculate the cosine measure for a given positive basis (or a given positive spanning set).

This paper provides a deterministic algorithm to compute the cosine measure of any given positive basis (or any finite positive spanning set). The  algorithm is proven to return the exact value of the cosine measure in finite time.  Hence, this paper  provides a procedure to compare positive bases to each other.

This paper is organized as follows. In Section \ref{Sec:Prel}, fundamental material on positive spanning sets, positive bases and background results are presented. In Section \ref{Sec:MyResults}, the algorithm and a proof that it returns the exact cosine measure for any finite positive spanning set are presented.  Section \ref{sec:complexity} examines the complexity of the algorithm and demonstrates that the algorithm can be shortened for minimal positive bases and maximal positive bases. Lastly, Section \ref{sec:conclusion} summarizes the main achievements of the paper and proposes some directions to explore in a near future.
\section{Preliminaries}\label{Sec:Prel}

In this paper, it will be convenient to regard a set of vectors as a matrix whose columns are the vectors in the set.  The vector space $\R^n$ is assumed for the entire paper.  The span of a set $\set$ in $\R^n$ is denoted $\spann(\set).$
\begin{definition}[Positive span and positive spanning set of $\R^n$]
The positive span of a finite set of vectors $\set=\begin{bmatrix}d_1& d_2&\cdots&  d_s \end{bmatrix}$ in $\R^n$, denoted $\pspan(\set)$, is the set $$\{v \in \R^n: v=\alpha_1d_1+ \dots+\alpha_s d_s, \alpha_i\geq 0, i=1, 2, \dots, s\}.$$
A finite positive spanning set of $\R^n$ of size $s$, denoted $\Pe^n_s,$ is a set of $s$ nonzero vectors such that $\pspan(\Pe^n_s)=\R^n.$
\end{definition}

 To define a positive basis of $\R^n$ requires the concept of \emph{positive independence}.
\begin{definition}[Positive independence]
A set of vectors $\set=\begin{bmatrix}d_1& d_2&\cdots&  d_s \end{bmatrix}$ in $\R^n$  is positively independent  if and only if $d_i \notin \pspan(\set\setminus {d_i})$ for all $i \in \{1, 2, \dots s\}.$
\end{definition}
\begin{definition}[Positive basis of $\R^n$]\label{def:pbasis}
A positive basis of $\R^n$ of size $s$, denoted $\D^n_s$, is a positively independent set of $s$ vectors whose positive span is $\R^n$.
\end{definition}

Equivalently, a positive basis of $\R^n$ can be defined as a set of nonzero vectors of $\R^n$ whose positive span is $\R^n,$ but for which no proper subset exhibits the same property.

The following theorem describes the structure of maximal positive bases.

\begin{theorem}\cite[Theorem 6.3]{Regis2016}\label{thm:maxbasis}
Suppose $\B=\begin{bmatrix}d_1&d_2&\cdots&d_n\end{bmatrix}$ is a basis of $\R^n.$ Then for any choice of $\delta_1, \dots, \delta_n>0$, the set $\D^n_{2n}=\begin{bmatrix}d_1&\cdots&d_n&-\delta_1 d_1&\cdots&-\delta_n d_n \end{bmatrix}$ is a positive basis of $\R^n.$  Conversely, every maximal positive basis of $\R^n$ has the form $(\D')^n_{2n}=\begin{bmatrix}d_1&\cdots&d_n&-\delta_1 d_1&\cdots&-\delta_n d_n \end{bmatrix}$ up to reordering of the vectors, where $\B'=\begin{bmatrix}d_1&\cdots&d_n\end{bmatrix}$ is a basis of $\R^n$ and $\delta_1, \dots, \delta_n>0.$
\end{theorem}
\begin{proposition}\cite[Theorem 2.3]{Regis2016} \label{Prop:Removing}
If $\set=\begin{bmatrix}d_1&d_2&\cdots&d_s \end{bmatrix}$ positively spans $\R^n$, then $\set \setminus\{d_i\}$ linearly spans $\R^n$ for any $i\in\{1, \dots, s\}$.
\end{proposition}

The \emph{cosine measure} of a set, the \emph{cosine vector set} and the \emph{active set} is defined next.

\begin{definition}[Cosine Measure]\label{def:cosinemeasure} The cosine measure of a positive spanning set $\Pe^n_s$  is defined by $$\cm(\Pe^n_s)=\min_{\substack{\Vert u \Vert=1\\u \, \in \, \R^n}} \max_{d \, \in \, \Pe^n_s} \frac{u^\top d}{\Vert d \Vert}.$$
\end{definition}

\begin{definition}[The cosine vector set]
Let $\Pe^n_s$ be a positive spanning set of $\R^n.$ The cosine vector set of $\Pe^n_s$, denoted $\V(\Pe^n_s)$, is defined as $$\V(\Pe^n_s)=\argmin{\substack{\Vert u \Vert=1\\u \, \in \, \R^n}}\max_{d\, \in\, \Pe^n_s} \frac{u^\top d}{\Vert d \Vert}.$$
\end{definition}
\begin{definition}[The active set of vectors]
Let $\Pe^n_s$ be a positive spanning set  of $\R^n$ and let $\us \in \V(\Pe^n_s).$ The active set of $\us$ in $\Pe^n_s,$ denoted $\A(\us,\Pe^n_s)$, is defined as $$\A(\us,\Pe^n_s)=\left \{\frac{d^\top}{\Vert d \Vert} \in \Pe^n_s:\frac{d^\top \us}{\Vert d \Vert}=\cm(\Pe^n_s) \right \}.$$
\end{definition}


Note that given any finite positive spanning set $\Pe^n_s$ where $n\geq 2$, the cosine measure is bounded by $0< \cm(\Pe^n_s)<1.$ To prove these bounds, next recall a theorem that helps to prove the lower bound.
\begin{theorem}\cite[Theorem 2.3]{Conn2009} \label{Theorem:DotProd}
Let $\set=\begin{bmatrix}d_1&d_2&\cdots&d_s\end{bmatrix}$ be a set of nonzero vectors in $\R^n$. Then $\set$ is a positive spanning set of $\R^n$ if and only if the following holds:
\begin{itemize}
    \item[i.] For every nonzero vector $v$ in $\R^n$, there exists an index $i \in \{1,2, \dots, s\}$ such that $v^\top d_i<0.$
    \item[ii.] For every nonzero vector $w$ in $\R^n$, there exists an index $j \in \{1,2, \dots, s\}$ such that $w^\top d_j>0.$
\end{itemize}
\end{theorem}
\begin{proposition} \label{Prop:CosMeasure}

Let $\Pe^n_s=\begin{bmatrix}d_1&d_2&\cdots&d_s\end{bmatrix}$ be a finite positive spanning set of $ \R^n$ (where $n\geq 2$). Then the cosine measure of $\Pe^n_s$ is bounded by $$0<\cm(\Pe^n_s)<1.$$
\end{proposition}
\begin{proof}
Without loss of generality, assume that $d_i$ are unit vectors for all $i \in \{1, 2, \dots, s\}$. Since $\Pe^n_s$ is a positive spanning set of $\R^n$, the cosine measure of $\Pe^n_s$ must be positive by Theorem \ref{Theorem:DotProd}.

Consider the upper bound. Since $\Pe^n_s$ is a finite set of vectors, there exist a nonzero  unit vector $u$ such that $u\neq d_i$ for all $i \in \{1, 2, \dots, s\}.$ Since the dot product of unit vectors is equal to 1 if and only if the two vectors are equal, it follows that $u^\top d_i<1$ for all $i \in \{1, 2, \dots, s\}$ and so  $\max_{d \, \in \, \Pe^n_s}u^\top d<1.$ Therefore the cosine measure  $\cm(\Pe^n_s)<1.$
\qed
\end{proof}

Note the fact that $\Pe^n_s$ is a positive spanning set is not used for the upper bound, only the fact that $\Pe^n_s$ is a finite set. Lastly, the definition of \emph{Gram matrices} and two lemmas that are helpful in Section \ref{Sec:MyResults} are introduced.

\begin{definition}[Gram matrix]
Let $\set=\begin{bmatrix}d_1&d_2& \cdots& d_s\end{bmatrix}$ be vectors in $\R^n$ with dot product $d_i^\top d_j$. The Gram matrix of the vectors $d_1, d_2, \dots, d_s$ with respect to the dot product, denoted $\G(\set),$ is an $s \times s$ real matrix where the entry $\G(\set)_{i,j}$ is defined as $\G(\set)_{i,j}=d_i^\top d_j, i,j \in \{1, 2, \dots, s\}$.
\end{definition}

\begin{lemma}\cite[Lemma 1]{Naevdal2018} \label{Lem:gamma} Let $\B=\begin{bmatrix}d_1&d_2&\cdots& d_n\end{bmatrix}$ be a basis of unit vectors in $\R^n$. Let $\one \in \R^n$ be the vector having all its entries equal to one. Then there exists a unit vector $u_\B \in \R^n$ such that $u_\B^\top d_i=\gamma_\B>0$ for all $i \in \{1, 2, \dots, n\}$ where $$\gamma_\B=\frac{1}{\sqrt{\one^\top \G(\B)^{-1}\one}}.$$
\end{lemma}

Note that the unit vector $u_\B$ such that $u_\B^\top d_i=\gamma_{\B}$ for all $i$'s is unique since $\{d_1, \dots, d_n\}$ is a set of $n$ linearly independent vectors. Also, note that $\gamma_\B<1$ whenever $n\geq 2.$

In fact, the positive value of the $n$ equal dot products is unique whenever $u$ is a unit vector and $\{d_1, \dots, d_n\}$ is a set of linearly independent vectors.

\begin{lemma} \label{lem:gammab}
Let $\B=\begin{bmatrix}d_1&d_2&\cdots& d_n\end{bmatrix}$ be a basis of unit vectors in $\R^n$. Suppose $u$ is a unit vector such that $u^\top d_1=\dots=u^\top d_n=\alpha>0.$ Then $\alpha=\gamma_\B,$ where $\gamma_\B$ is defined in Lemma \ref{Lem:gamma}.
\end{lemma}
\begin{proof}
It suffices to show $\alpha>0$ is unique, so that Lemma \ref{Lem:gamma} implies it must be the value $\gamma_\B.$

Suppose there exists two distinct unit vectors, say $u$ and $u'$ such that
 $$u^\top d_1=\dots= u^\top d_n=\alpha >0 ~\mbox{and}~ u'^\top d_1=\dots= u'^\top d_n=\alpha' >0.$$  Since $\B$ is a basis, it follows that $\alpha \neq \alpha'$ and there exists $\rho_1, \dots, \rho_n \in \R$ such that $\sum_{i=1}^n \rho_i d_i = u$.  Multiplying both sides by by $u'^\top$, shows that $\sum_{i=1}^n \rho_i \alpha' = u'^\top u$.  Alternately, multiplying both sides by $u^\top$ yields $\sum_{i=1}^n \rho_i \alpha=1$.  Letting $\bar{\rho}=\sum_{i=1}^n\rho_i$, provides
     $$\bar{\rho}\alpha'=u^\top u' ~\mbox{and}~ \bar{\rho}\alpha=1.$$
 Similarly, since $\B$ is a basis, there exist $\beta_1, \dots, \beta_n \in \R$ such that $\sum_{i=1}^n \beta_i d_i = u'$.  Letting $\bar{\beta}=\sum_{i=1}^n\beta_i$ and multiplying this by $u^\top$ and $u'^\top$, yields
    $$\bar{\beta}\alpha=u^\top u' ~\mbox{and}~ \bar{\beta}\alpha'=1.$$
Applying $\bar{\rho} = 1/\alpha$ and $\bar{\beta}=1/\alpha'$ into the second equality yields 	$\frac{\alpha'}{\alpha} = \frac{\alpha}{\alpha'}.$
Since $\alpha > 0$ and $\alpha'>0$, this yields $\alpha=\alpha'$.
\qed
\end{proof}

\section{Main results}\label{Sec:MyResults}
In this section, an algorithm that calculates the cosine  measure for any finite positive spanning set $\Pe^n_s$ of $\R^n$ (or any positive basis $\D^n_s$ of $\R^n$) is provided. After introducing the algorithm, it is shown that the algorithm returns the exact value of the cosine measure.
\begin{center}
\scalebox{1}{
\begin{algorithm}[H]
\DontPrintSemicolon
\caption{The cosine measure of a finite positive spanning set of $\R^n$ \label{alg:cmpbasis}}
Given $\Pe^n_s$, a finite positive spanning set of $\R^n$:

         \textbf{1. For all bases $\B\subset \Pe^n_s$, compute\;  }
         {
  $ \displaystyle
    \begin{aligned}
      (1.1)&\quad\gamma_\B=\frac{1}{\sqrt{\one^\top \G(\B)^{-1}\one}} &&\text{(The positive value of the $n$ equal dot products)},\\
      (1.2)&\quad u_\B=\B^{-\top}\gamma_\B \one &&\text{(The unit vector associated to $\gamma_{\B}$)},\\
      (1.3)&\quad p_\B=\begin{bmatrix}p_{\B}^1&\cdots&p_{\B}^s \end{bmatrix}=u_\B^\top\Pe^n_s &&\text{(The dot product vector)}, \\
      (1.4)&\quad\p =\max_{1\leq i \leq s }p_\B^i &&\text{(The maximum value in $p_{\B}$)}.
    \end{aligned}
  $
\par}

        \textbf{2. Return  \;}
    	
    	 {
  $ \displaystyle
    \begin{aligned}
        (2.1)&\quad \cm(\Pe^n_s)=\min_{\B \, \subset \, \Pe^n_s} \p &&\text{(The cosine measure of $\Pe^n_s$)}\\
        (2.2)&\quad\V(\Pe^n_s)=\{u_{\B}:\p=\cm(\Pe^n_s)\}&&\text{(The cosine vector set of $\Pe^n_s$)}.
    \end{aligned}
  $
\par}
\end{algorithm}}
\end{center}

The algorithm investigates all the bases contained in $\Pe^n_s.$ Note that any finite positive spanning set of $\R^n$ contains at least one basis of $\R^n$ (by Proposition \ref{Prop:Removing}). An obvious upper bound for the maximal number of bases contained in a finite positive spanning set $\Pe^n_s$ is ${s \choose n}$. The precision of these bounds may be improved, but it is beyond the scope of this paper. However, the supremum of the number of bases contained  in minimal positive bases ($s=n+1$) and maximal positive bases ($s=2n$) are easily derived.
\begin{proposition} Let $\D^n_{n+1}$ be a minimal positive basis of $\R^n$. Then $\D^n_{n+1}$ contains $n+1$ bases of $\R^n.$
\end{proposition}
\begin{proof}From Proposition \ref{Prop:Removing}, any set of $n$ vectors is a  basis of $\R^n$. Therefore, $\D^n_{n+1}$ contains ${n+1 \choose n}=n+1$ bases. \qed
\end{proof}
\begin{proposition}\label{prop:iterationsformaximal} Let $\D^n_{2n}$ be a maximal positive basis. Then $\D^n_{2n}$ contains $2^n$ bases.
\end{proposition}
\begin{proof}
Without loss of generality, by Theorem \ref{thm:maxbasis}, let $\D^n_{2n}=\begin{bmatrix}d_1&d_2&\cdots&d_n&-d_1&-d_2&-d_n\end{bmatrix},$ where $d_i$ is a unit vector for all $i \in \{1, 2, \dots, n\}$. Note that any basis contained in $\D^n_{2n}$ has the form $\B=\begin{bmatrix}\pm d_1&\pm d_2&\cdots&\pm d_n \end{bmatrix}$. Therefore, $\D^n_{2n}$ contains $2^n$  bases. \qed
\end{proof}

Note that  $2^n$ is  smaller than ${2n \choose n}$ whenever $n \in \{2, 3, \dots\}.$
Finding the supremum for the number of bases contained in a positive basis of intermediate size ($n+1<s<2n$) is more challenging and is left for future exploration. Nevertheless, the number of bases  in any positive  basis $\D^n_s$ (or any finite positive spanning set $\Pe^n_s$) is a finite number greater than one and hence, the algorithm always find an exact solution in finite time.

To prove that the Algorithm \ref{alg:cmpbasis} returns the  exact cosine measure of a positive spanning set for any size $s \in \{2, 3, \dots\}$ requires the following lemma.

\begin{lemma}\label{lem:epsover2}
Let $\epsilon \neq 0$ and let $u$ and $v$ be unit vectors in $\R^n$. Then
    \begin{itemize}
        \item[i.] $\Vert u+\epsilon v \Vert=1$ if and only if $\epsilon=-2u^\top v$, and
        \item[ii.] $\Vert u+\epsilon v \Vert<1$ implies $\Vert u-\epsilon v \Vert>1$.
    \end{itemize}
\end{lemma}
\begin{proof}
Since
   $$
    \Vert u +\epsilon v \Vert=1 + (2\epsilon u^\top v + \epsilon^2) ~\mbox{and}~\Vert u -\epsilon v \Vert=1 + (2\epsilon u^\top v - \epsilon^2),
   $$
it follows that $\Vert u +\epsilon v \Vert=1$ if and only if $2\epsilon u^\top v + \epsilon^2=0$.  Since $\epsilon \neq 0$, the first result follows.

Considering $\Vert u+\epsilon v \Vert<1$, notice that
$$\Vert u+\epsilon v \Vert<1 \iff 1 + (2\epsilon u^\top v + \epsilon^2)<1
\iff 2\epsilon u^\top v + \epsilon^2 < 0
\iff 2\epsilon u^\top v - \epsilon^2 < -2\epsilon^2 < 0,$$
which implies the second result.
\qed
\end{proof}

The previous lemma is used in the following proposition.
\begin{proposition} \label{prop:spanrn}
Let $\Pe^n_s$ be a positive spanning set of $\R^n$ and let $\us \in \V(\Pe^n_s)$.
Then $$\spann(\A(\us, \Pe^n_s))=\R^n.$$
\end{proposition}
\begin{proof}
Without loss of generality, assume that all vectors $d$ in $\Pe^n_s$ are unit vectors.
Suppose that $\spann(\A(\us, \Pe^n_s))\neq\R^n$, i.e., the rank of $\A(\us, \Pe^n_s)$ is strictly less than $n.$ This implies that the kernel of $\A(\us, \Pe^n_s)$ is nonempty. Let $v$ be a unit vector in the kernel of $\A(\us, \Pe^n_s)$. This means that $d^\top v=0$ for all $d$ in $\A(\us,\Pe^n_s).$

Notice that, if $d \in \Pe^n_s\setminus \A(\us,\Pe^n_s),$ then $$d^\top \us <\cm(\Pe^n_s).$$
 Consider the vector $\us+\epsilon v.$ Then there exists an $\epsilon$ such that $0<\epsilon<\vert -2\us^\top v\vert$ and  $$\frac{d^\top (\us\pm \epsilon v)}{\Vert \us \pm \epsilon v \Vert}<\cm(\Pe^n_s)$$ for all $d \in \Pe^n_s\setminus \A(\us, \Pe^n_s).$  Moreover, since $d^\top v=0,$ it follows that $$\frac{d^\top(\us\pm \epsilon v)}{\Vert\us \pm \epsilon v \Vert}=\frac{d^\top \us}{\Vert\us \pm \epsilon v \Vert} \pm 0=\frac{\cm(\Pe^n_s)}{\Vert \us \pm \epsilon v \Vert}$$  for all $d \in \A(\us, \Pe^n_s).$
By Lemma \ref{lem:epsover2}(i), $\epsilon \neq -2\us^\top v$ implies that $\Vert \us + \epsilon v \Vert \neq 1.$ By Lemma \ref{lem:epsover2}(ii), if $\Vert \us +\epsilon v \Vert<1$, then $\Vert \us -\epsilon v \Vert>1.$ Select $w$  in $\{\us +\epsilon v, \us -\epsilon v \}$ such that $\Vert w \Vert >1.$ Then $$\frac{d^\top w}{\Vert w \Vert}< \cm(\Pe^n_s)$$ for all $d \in \Pe^n_s.$ This contradicts the definition of cosine measure.

Therefore, $\spann(\Pe^n_s)\subseteq \spann(\A(\us, \Pe^n_s))$.  Since $\spann(\Pe^n_s)=\R^n$, the result follows. \qed
\end{proof}

Note that the positive spanning set property of $\Pe^n_s$ in the previous proposition is sufficient to prove the result. The positive independence property of positive bases is not necessary to obtain the result.  This provide sufficient background to complete the proof that Algorithm \ref{alg:cmpbasis} returns the exact cosine measure of any  finite positive spanning set of $\R^n.$


\begin{corollary}\label{cor:containbasis}

Let $\Pe^n_s$ be a finite positive spanning set of $\R^n$ and let $\us \in \V(\Pe^n_s).$ Then
$\A(\us, \Pe^n_s)$ contains a basis of $\R^n.$
\end{corollary}
 This is a classical result in linear algebra. See \cite[Theorem 2.11]{Brown1988} for example.

\begin{theorem} \label{thm:mainresult}
Let $\Pe^n_s$ be a finite positive spanning set of $\R^n.$  Then Algorithm \ref{alg:cmpbasis} returns the exact value of the  cosine measure $\cm(\Pe^n_s).$

\end{theorem}
\begin{proof}
Without loss of generality, let $\Pe^n_s=\begin{bmatrix}d_1&d_2&\cdots&d_s\end{bmatrix}$ be a  finite positive spanning set of unit vectors in $\R^n$ and let $\us \in \V(\Pe^n_s).$ By Corollary \ref{cor:containbasis}, $\A(\us, \Pe^n_s)$ contains a basis of $\R^n.$ Without loss of generality, let this basis be $\B_*=\begin{bmatrix}d_1&d_2&\cdots&d_n\end{bmatrix}.$ So that $$\cm(\Pe^n_s)=d_1^\top \us=\dots=d_n^\top \us>0,$$ where $\us$ is a unit vector.  By Lemma \ref{lem:gammab}, it follows that $$\cm(\Pe^n_s)=\gamma_{\B_*}=\frac{1}{\sqrt{\one^\top \G(\B_*)^{-1} \one}}.$$
Note that $\gamma_{\B_*}=\max_{1\leq i \leq s}d_i^\top \us=\mathring{p}_{\B_*}$ since $\gamma_{\B_*}=\cm(\Pe^n_s).$ Therefore, by definition of the cosine measure, $$\min_{\B \, \subset \, \Pe^n_s}\p=\cm(\Pe^n_s),$$  where $\B$ is a basis of $\R^n$ contained in $\Pe^n_s.$ \qed
\end{proof}

\section{Complexity}\label{sec:complexity}

The complexity of an algorithm is a count of the number of floating point operations (flops) required to complete the algorithm.   As noted above, the maximum number of iterations required by the algorithm for a finite positive spanning set $\Pe^n_s$ is ${s \choose n}$; unless a maximal positive basis is imputed, in which case the required number of iterations is $2^n$ (Proposition \ref{prop:iterationsformaximal}).  The complexity in big-oh notation per iteration is next.

\begin{proposition} Let $\Pe^n_s$ be a finite positive spanning set of $\R^n$. Then Algorithm \ref{alg:cmpbasis} has a complexity of  $O(s^2 n)+ O(s^3) + O(n^3)$ flops per iteration (assuming basic matrix inversion techniques).
\end{proposition}

\begin{proof}
Computing the Gram matrix $\G(\B)$  requires $O(s^2 n)$ flops.  Using basic methods, the matrix inversion of the Gram matrix uses $O(s^3)$ flops.  The matrix multiplication in step (1.1) is $O(2n^2)$ and the square roots and division are negligible.    The matrix $\B$ is $n \times n$, so inversion is $O(n^3)$.  Matrix multiplication in step (1.3) is $O(n^2)$ and the maximum in step (1.4) is negligible.  All of the operations in step 2. are negligible.  So, the major effort is the construction of the Gram matrices and the matrix inversions, resulting in $O(s^2 n)+ O(s^3) + O(n^3)$ flops per iteration.
\end{proof}

Note, the complexity above could be improved slightly if more advanced matrix inversion methods are used \cite{Golub96}.  However, the complexity of constructing the Gram matrix will remain $O(s^2 n)$, so little is gained by doing this.

Algorithm \ref{alg:cmpbasis} can be shortened for minimal positive bases ($s=n+1$) and maximal positive bases ($s=2n$).

\begin{theorem}Let $\D^n_{n+1}=\begin{bmatrix}d_1&d_2&\cdots&d_{n+1}\end{bmatrix}$ be a  minimal positive basis of $\R^n.$ Then $$\gamma_{\B}=\p$$ for all  bases $\B \subset \D^n_{n+1}$ where $\gamma_\B$ and $\p$ are defined as in Algorithm \ref{alg:cmpbasis}. Moreover, $$\cm(\D^n_{n+1})=\min_{\B \, \subset \, \D^n_{n+1}} \gamma_\B.$$
\end{theorem}
\begin{proof}Let $\B$ be a basis of $\R^n$ contained in $\D^n_{n+1}.$ Since $\gamma_{\B}>0,$  if follows that $d^\top u_{\B}<0$ (by Theorem \ref{Theorem:DotProd}) where $d$ is the only vector in $\D^n_{n+1}\setminus \B.$ Therefore, $\p=\gamma_\B$ for all bases $\B \subset \D^n_{n+1}$ and it follows that $$\cm(\D^n_{n+1})=\min_{\B \, \subset \, \D^n_{n+1}} \gamma_\B.$$
\end{proof}


\begin{theorem}Let $\D^n_{2n}=\begin{bmatrix}d_1&d_2&\cdots&d_{2n}\end{bmatrix}$ be a  maximal positive basis of $\R^n.$ Then $$\gamma_\B=\p$$ for all  bases $\B \subset \D^n_{2n}$ where $\gamma_\B$ and $\p$ are defined as in Algorithm \ref{alg:cmpbasis}. Moreover, $$\cm(\D^n_{2n})=\min_{\B \, \subset \, \D^n_{2n}} \gamma_\B.$$
\end{theorem}
\begin{proof}
Without loss of generality, by Theorem \ref{thm:maxbasis}, let $\D^n_{2n}=\begin{bmatrix}d_1&\cdots&d_n&-d_1&\cdots&-d_n\end{bmatrix}$ be a positive basis of unit vectors for $\R^n$.
Note that every basis  contained in $\D^n_{2n}$ has the form $\begin{bmatrix}\pm d_1& \pm d_2&\dots&\pm d_n\end{bmatrix}$. Hence, without loss of generality  relabelling if necessary, let $\B=\begin{bmatrix}d_1&\cdots&d_n \end{bmatrix}.$ So $$u_{\B}^\top d_1=\dots=u_{\B}^\top d_n=\gamma_{\B}>0.$$
It follows that $u_{\B}^\top(-d_i)<0$ for all $i \in \{1, \dots, n\}.$ Therefore, $\gamma_\B=\p$
for all bases $\B$ contained in $\D^n_{2n}$ and it follows that $$\cm(\D^n_{2n})=\min_{\B \, \subset \, \D^n_{2n}} \gamma_\B.$$ \qed
\end{proof}

A consequence of the previous two theorems is that it is not necessary to compute $p_\B,$ and $\p$ in Algorithm \ref{alg:cmpbasis}. This means that step $(1.3)$ and step $(1.4)$ can be deleted from Algorithm \ref{alg:cmpbasis}.  The cosine measure (step $(2.1)$) and the cosine vector set (step $(2.2)$) can be found by  simply setting $$\cm(\D^n_s)=\min_{\B \, \subset \, \D^n_s} \gamma_\B$$ and $$\V(\D^n_s)=\{ u_{\B}:\gamma_{\B}=\cm(\D^n_s) \}$$ whenever $s=n+1$ or $s=2n.$  Unfortunately, this does not impact the complexity per iteration, as constructing the Graham matrices and the matrix inversions are still required.


The next example shows that the previous abridged algorithm does not guarantee to return the value of the cosine measure for positive bases of intermediate size ($n+1<s<2n$).

\begin{example}[Alg. \ref{alg:cmpbasis} cannot be shortened for all positive bases of intermediate size]

Let $$\D^3_{5}=\begin{bmatrix}1&0&0&-0.8&0\\0&1&0&0&-0.9\\0&0&1&-0.6&-\sqrt{0.18}\end{bmatrix}.$$ Then $\D^3_{5}$ is an intermediate positive basis of $\R^3.$  Computation shows that $$\min_{\B \, \subset \D^3_5}\gamma_\B\approx 0.2038.$$ and the unit vector associated to the minimal $\gamma_{\B}$ is  $u_\B\approx\begin{bmatrix}0.4115&0.2038&-0.8883\end{bmatrix}^\top$ where $$\B=\begin{bmatrix}0&-0.8&0\\1&0&-0.9\\0&-0.6&-\sqrt{0.19}\end{bmatrix}.$$ Computing $p_{\B}$ and $\p$, yields $p_{\B} \approx \begin{bmatrix}0.4115&0.2038&-0.8883&0.2038&0.2038\end{bmatrix}$ and so $$\p\approx0.4115 \neq \gamma_\B.$$
Note that  the cosine measure of $\D^3_{5}$ is found when considering $$\B_*=\begin{bmatrix}1&0&-0.8\\0&1&0\\0&0&-0.6\end{bmatrix}.$$ Thus, $\gamma_{\B_*} \approx 0.3015 $ and $u_{\B_*} \approx \begin{bmatrix}0.3015&0.3015&-0.9045\end{bmatrix}^\top.$ The dot product vector is $p_{\B_*}\approx\begin{bmatrix}0.3015&0.3015&-0.9045&0.3015&0.1229\end{bmatrix}$ and so $$\cm(\D^3_{5})=\mathring{p}_{\B_*}\neq \min_{\B \, \subset \D^3_5}\gamma_\B.$$
\end{example}

\section{Conclusion and open directions} \label{sec:conclusion}
This paper has presented a deterministic algorithm to compute the cosine measure of any positive basis in $\R^n.$  In fact, the algorithm can be applied to any finite positive spanning set of $\R^n.$ One weakness of the algorithm, and a topic to explore, is that the algorithm needs to investigate ${s \choose n}$ sets of vectors and decide if the set is a basis. Indeed, as $s$ increases, this number becomes  extremely large. Hence, creating a computationally inexpensive technique, by exploiting the structure of positive bases, to decide if the set of $n$ vectors is a basis, could speed up the algorithm significantly.

It was found that the algorithm can be slightly shortened when considering a minimal positive basis ($s=n+1$) or a maximal positive basis ($s=2n$).Thereafter, an example was provided demonstrating that the abridged version of the algorithm  is not  valid  for intermediate positive bases.

Moreover, it was showed that a maximal positive basis contained $2^n$ bases of $\R^n$ and a minimal positive basis contained $n+1$ bases of $\R^n.$ The maximal number of bases contained in a positive basis of intermediate size ($n+1<s<2n$) would be valuable to investigate. A better understanding of the structure of intermediate positive bases will certainly help to answer this question.

In 2018, it was rigorously showed, with concepts of matrix algebra, that a maximal positive basis ($s=2n$) has maximal cosine measure $1/\sqrt{n}$ and a minimal positive basis ($s=n+1$) has maximal cosine measure $1/n.$ The positive bases attaining these upper bounds have also been characterized \cite{Naevdal2018}. However, finding the positive basis of intermediate size  with maximal cosine measure is still an open question \cite{Dodangeh2016,Naevdal2018}. Hopefully the algorithm provided in this paper will be useful to answer this question. The algorithm presented in this paper could also be used to find, employing a numerical approach, the maximal cosine measure for a positive spanning set of $2n$ vectors (the maximal cosine measure for a positive basis of size $2n$ is $1/\sqrt{n}$, but this value could be wrong if we consider all the positive spanning sets of size $2n$ instead,  as discussed in \cite{Dodangeh2016}).

\normalsize
\bibliographystyle{siam}
\bibliography{Mybib}
\end{document}